\documentclass[12pt]{amsart}
\usepackage[margin=1.2in]{geometry}
\usepackage{color}
\title[dynamical sampling for source term recovery]{Recovery of rapidly decaying source terms from  dynamical samples in evolution equations }

\usepackage{hyperref}

\date{\vspace{-5ex}}
\usepackage{amsthm, amsfonts, mathrsfs}
\usepackage{ dsfont }
\usepackage{amssymb}
\usepackage{enumerate}
\usepackage{graphicx}
\usepackage{subfig}
\usepackage{float}
\usepackage{bbm}
\usepackage{amsmath}
\usepackage{comment}
\usepackage{hyperref}
\usepackage{listings}
\usepackage{color}

\usepackage{diagbox}
\usepackage[dvipsnames]{xcolor}
\usepackage{algorithm}
\usepackage{booktabs}
\usepackage{algorithmicx}
\usepackage[noend]{algpseudocode}
\usepackage{tikz}
\allowdisplaybreaks

\hypersetup{
    colorlinks,
    citecolor=blue,
    filecolor=blue,
    linkcolor=blue,
    urlcolor=blue
}

\newtheorem{theorem}{Theorem}[section]
\newtheorem{proposition}[theorem]{Proposition}
\newtheorem{lemma}[theorem]{Lemma}

\newtheorem{definition}[theorem]{Definition}
\newtheorem{remark}[theorem]{Remark}

\newtheorem{assume}{Assumption}

\newcommand{\HH}{\mathcal{H}}

\newcommand{\R}{\mathbb{R}}

\newcommand{\N}{\mathbb{N}}

\newcommand{\estime}{\mathfrak t}

\definecolor{aacolor}{rgb}{0.05, 0.75, 1}

\definecolor{ikcolor}{rgb}{1, 0., 0.}

\begin{document}
	\date{}

\author{Akram Aldroubi, Le Gong
		and Ilya Krishtal}
		
\address{Akram Aldroubi, Department of Mathematics, University of Vanderbilt}
\email{akram.aldroubi@vanderbilt.edu}

\address{Le Gong, Department of Mathematics, University of Vanderbilt}
\email{le.gong@vanderbilt.edu}

\address{Ilya Krishtal, Department of Mathematical Sciences, Northern Illinois University}
\email{ikrishtal@niu.edu}
%
%\author{\IEEEauthorblockN{Akram Aldroubi\IEEEauthorrefmark{1},
%		Longxiu Huang\IEEEauthorrefmark{2}, Keri Kornelson\IEEEauthorrefmark{3}, and
%		Ilya Krishtal\IEEEauthorrefmark{4}}
%\IEEEauthorblockA{\IEEEauthorrefmark{1}\IEEEauthorrefmark{2}\IEEEauthorrefmark{3}\IEEEauthorrefmark{4}Department of Mathematics,\\
%		\IEEEauthorrefmark{1}\IEEEauthorrefmark{2}Vanderbilt University, \IEEEauthorrefmark{3}University of Oklahoma, \IEEEauthorrefmark{4}Northern Illinois University\\
%		\IEEEauthorrefmark{1}\IEEEauthorrefmark{2}Nashville, TN 37240, U.S.A.,
%		\IEEEauthorrefmark{3}Norman, OK 73019, U.S.A.
%		\IEEEauthorrefmark{4}DeKalb, IL 60115, U.S.A.\\
%		Email: \IEEEauthorrefmark{1}akram.aldroubi@vanderbilt.edu,
%		\IEEEauthorrefmark{2}longxiu.huang@vanderbilt.edu,
%		\IEEEauthorrefmark{3}kkornelson@ou.edu,
%		\IEEEauthorrefmark{4}ikrishtal@niu.edu}}
%


\maketitle
\begin{abstract}
We analyze the problem of recovering  a source term of the form $h(t)=\sum_{j}h_j\phi(t-t_j)\chi_{[t_j, \infty)}(t)$  %that drives the evolution of a signal 
from space-time samples of the solution $u$ of an  initial value problem in a Hilbert space of functions. In the expression of $h$, the terms $h_j$ belong to the Hilbert space, while $\phi$ is a generic real-valued function with exponential decay at $\infty$. The design of the sampling strategy takes into account noise in measurements and the existence of a background source. 
\end{abstract}


% \fi

\section{Introduction}
%\leg{test}

Dynamical sampling refers to a set of problems in which a space-time signal $u$ evolving in time under the action of a linear operator as in \eqref {DFM} below is to be sampled on a space-time set $S=\{(x,t): x \in X, t\in \mathcal T\}$ in order to recover $u_0$, $u$, $F$ or other information  related to these functions. For example, when the goal is to recover $u_0$, % from space-time samples $\{u(x,t): (x,t) \in S \}$, 
we get the so called space-time trade-off problems  (see e.g., \cite{APT15, ADK13, ADK15,  AGHJKR21, CMPP20, CJS15,  DMM21, UZ21, ZLL17}). If the goal is to recover the unknown underlying operator $A$, or some of its spectral characteristics, we get the system identification problem in dynamical sampling \cite{AHKLLV18, AK16,  Tan17}. In other situations, the goal is to identify the driving source term from space-time samples \cite{AHKK21,  AK15}.  In all dynamical sampling problems, frame theory plays a fundamental albeit, at times, hidden role (see e.g. \cite{ACCP21, AHKLLV18, AHP19, CMPP20}).  Moreover, this important connection has also been used to develop frame theory and led to the  concept of dynamical frames (see e.g.  \cite{AR18, AKh17, BI221,  CH18, CHP20, Cor21, KS19}).  In this paper, we consider the problem of designing space-time sampling patterns that permit  recovery of the source term 
of an initial value problem (IVP) 
%$F$ in \eqref {DFM} 
or some relevant portion thereof. 

 \subsection{Motivation} In \cite{AHKK21}, the authors introduced a new sampling technique which prescribes how one may sample the solution of an
 %initial value problem (IVP) 
 IVP to detect ``bursts'' in the driving force of the system. The proposed special structure of the samplers allowed one to ``predict'' the value of the solution at the next sampling instance provided that no burst occurred during the sampling period. Thus, if the samples at the end of the period were significantly different from the prediction, a ``burst'' must have occurred. In \cite{AHKK21}, the ``bursts'' were modeled as a linear combination of Dirac measures. 
In this paper, we employ a modification of the same technique to detect %``bursts'' that are modeled by a different class of functions,  %The work is motivated by the estimate of 
localized non-instantaneous sources which decay exponentially in time after activation \cite{MBD15}. %The class of functions modeling is commonly used when modeling 
Such sources may describe, for example, an irregular intake of rapidly degrading substances. In particular, the IVP we consider may model a complex chemical reaction contaminated by such an intake, and our goal, in this case, would be to determine when and what substances were added to the system. Many other phenomena driven by natural mechanisms, such as the dispersion of pollution, the spreading of fungal diseases and the leakage of biochemical waste, can also be described by IVP with the source terms considered here (see \cite{MBD15, MBD17, RCLV11} and references therein). 
Thus, a robust sampling and reconstruction algorithm for such IVP would be beneficial for studying these real-world applications. 

%The dispersion of pollution, the spreading of fungal diseases and the leakage of biochemical waste are typical examples governed by diffusion equations. Besides, numerous examples governed by wave and Laplace's equations also attract considerable research attention (see \cite{MBD15, MBD17, RCLV11} and references therein).  

 \subsection{Problem setting}
 Let us give a more precise description of our setting.
We consider the following abstract initial value problem:
\begin{equation}\label{DFM}
	\begin{cases}
	\dot{u}(t)=Au(t)+F(t)\\
	u(0)=u_0,
	\end{cases}
	\quad t\in\mathbb R_+,\ u_0\in\HH.
\end{equation} 
Above, the variable $x \in \R^d$  is the ``spatial''  variable and $t\in\R_+$ represents time.
For each fixed $t\in\mathbb R_+$, %$u(\cdot,t)$ can be  viewed as a vector 
$u(t)$ is a vector in some Hilbert space $\HH$ of functions on a subset of $\R^d$. 
Additionally, $\dot{u}: \R_+\to\HH$ is the time derivative of $u$, $F:\mathbb{R}_+\rightarrow\HH$ is  a forcing term a portion of which we wish to recover, and $A: D(A)\subseteq \HH\to\HH$ is a generator of a strongly continuous semigroup $T:\mathbb R_+\to B(\HH)$.

As in \cite{AHKK21}, we consider force terms of the form:
\[
F=h+\eta
\]
where $\eta$ is a Lipschitz continuous background source term. Unlike \cite{AHKK21}, however,  the burst-like forcing term $h$ is assumed to be given by 
\begin{equation}\label{Force}
h(t)=\sum_{j=1}^Nh_j\phi(t-t_j)\chi_{[t_j, \infty)}(t),
\end{equation}
where $0<t_1 < \ldots < t_N$, $h_j\in\HH$, and   $\phi$ is a non-negative function with a certain prescribed decay on $[0,\infty).$ We regard $t_j$ and $h_j$ as the time and the shape of the  $j$-th burst, respectively.

The goal of this paper is to provide an algorithm similar to one in \cite{AHKK21} that recovers the ``burst-like'' portion $h$ of $F$ from space-time samples  %$\{u(x_i,t_k)\}$ 
of the solution $u$ of \eqref{DFM}. Once again we shall choose the structure of the samplers that would allow one to detect an occurrence of a burst in a sampling period by comparing the predicted values of the samples with the actual samples. In fact, we will show that the same structure of samplers that was used in the first of the two approaches in \cite{AHKK21} may also be used in the current situation. The recovery algorithms in this paper, however, are significantly different and are not just a straightforward tweak of the ones in \cite{AHKK21}. The main difficulty in the current setting is the need to account for the influence of past bursts which was not an issue when those were Diracs.

%a function $u=u(x,t)$ that evolves in time due to the action of a known evolution operator $A$ and the forcing function $F$. 

\subsection{Paper organization}
The rest of the paper is organised as follows. In Section \ref{sec2}, we remind the reader of some basic properties of one-parameter operator semigroups and their use in solving IVP such as \eqref{DFM}. We also list all model assumptions for the algorithms of this paper. 

In Section \ref{sec3}, we present the main results. The section is divided into two parts. In Subsection \ref{model1}, the decay function $\phi$ in \eqref{Force} is assumed to be of the form $e^{-\rho t}$ for some $\rho>0.$ We present the structure of  measurement functions (\ref{measurements}) for this case and utilize discrete samples of the measurement functions to approximate the burst time and shape in the presence of background source and measurement acquisition errors. In Subsection \ref{model2}, we consider a more general model, where the decay function $\phi$ does not have a concrete formula, but is rather bounded above by a decaying exponential function. Under additional assumptions that the shapes of the bursts are uniformly bounded and that the differences $t_{j+1}-t_j$, $j=1,...,N-1$, are large enough, we present a modification of the algorithm from Subsection \ref{model1}, which solves the same problem in this more general setting.
Finally, in Section \ref{sec4}, we set up specific (synthetic) dynamical systems to test the algorithms and describe the results of the testing.

\section{IVP solution and model assumptions}\label{sec2}

In this section, we recall a few basic facts from the theory of one-parameter operator semigroups and summarize our model assumptions. 

\subsection{IVP toolkit}

\begin{definition}
A strongly continuous operator semigroup  is a map $T:\mathbb{R}_+\rightarrow B(\HH)$ (where $B(\HH)$ is the space of all bounded linear operators on $\HH$), which satisfies
\begin {enumerate}[ (i)]
\item $T(0)=I$, 
\item $T(t+s)=T(t)T(s)$ for all $t,s\geq 0$, and
\item $\|T(t)h-h\|\rightarrow 0$ as $t\rightarrow 0$ for all $h\in \HH$.
\end {enumerate}
\end{definition}

\begin{proposition}\label{pp1}\cite{EN00}
There exist constants $a\in\R$ and $M\ge 1$ such that
\[||T(t)||\le Me^{at}\] for all $t\ge 0.$
\end{proposition}

Recall \cite[p.~436]{EN00} that the (mild) solution of \eqref{DFM} can be represented  as
\begin{equation}\label{SfDFM}
	u(t)=T(t)u_0+\int_{0}^{t}T(t-s)F(s)ds.
\end{equation}
Substituting $F=h+\eta$ with $h$ of the form \eqref{Force} yields
\begin{equation}\label{solburst}
	u(t)=T(t)u_0+\sum\limits_{t_j<t}\int_{t_j}^{t}T(t-s)h_j\phi(s-t_j)ds+\int_0^{t}T(t-s)\eta(s)ds, \quad
	t\ge 0.
\end{equation}

In this paper, we will use the measurement function $\mathfrak m:\R_+\times \HH\to \R$ given by  
 \begin{equation}\label{measurements}
	\mathfrak m(t,g) = \left\langle u(t),g \right\rangle +\nu(t,g),\ t\ge 0,\ g\in G,
\end{equation}
where $\left\langle \cdot,\cdot \right\rangle$ is the inner product in $\HH$, $\nu$ is the measurement acquisition noise, and $G$ is the collection of samplers whose structure we wish to prescribe. 

\subsection{Model assumptions}

\begin{assume}\label{as2}
The set of samplers $G$ has the form $G = \widetilde G \cup T^*(\beta) \widetilde G$ for some countable (possibly, finite) set $\widetilde G\subseteq \HH$. Additionally, in the model of Subsection \ref{model2}, the set $\widetilde G$ is assumed to be uniformly bounded by some $R \in \R$, i.e. $R = \sup_{g\in\widetilde G} \|g\| < \infty$.
\end{assume}

\begin{assume}\label{as3}
In Subsection \ref{model2}, the burst terms are uniformly bounded, i.e. $\|h_j\|\le H$, $j=1, \ldots, N$, for some $H\in\R$.
%\ik{we might need it elsewhere as well, to be able to transfer the effect of previous bursts to the Lipschitz background.}
\end{assume}

\begin{assume}\label{as4}
The background source $\eta: \R_+\to\HH$ is uniformly bounded by a constant $K>0$ and Lipschitz with a Lipschitz constant $L\ge 0$, i.e.   $\sup\limits_{t\ge0}\|\eta(t)\|\le K$ and $\|\eta(t+s)-\eta(t)\|\le Ls$, $t,s\in \R_+$.
\end{assume}

\begin{assume}\label{as5}
The additive noise  $\nu$ in the measurements 
\eqref {measurements}  satisfies $$\sup\limits_{t > 0,\ h\in\HH}|\nu(t,h)|\le \sigma.$$

\end{assume}

\begin{assume}${}$\\ \label{as6}
%We assume that no bursts occurred in the time interval $[0, \beta)$ \ak{Is it really necessary?What we should say, is that in our theory we cannot assess what happens in the interval $[0, \beta)$. Is that right? However this is not an assumption}. 
\begin{enumerate} 
    \item In Section \ref {model1} we assume that the distance between two bursts is bounded below: $t_{j+1}-t_j\ge 4\beta$. 
    \item In Section \ref {model2}  we assume $t_{j+1}-t_j\ge D+4\beta$  with some positive number $D$.  
\end{enumerate}
\end{assume}

\section{Main Results}\label{sec3}
 
\subsection{Model with specific decay function}\label{model1}
In this section, we consider the special case where the decay function $\phi$ is given by $\phi(t)=e^{-\rho t}$ with some $\rho>0$:
\begin{equation}
\begin{cases}
\dot{u}(t)=Au(t)+\sum\limits_j^Nh_je^{-\rho(t-t_j)}\chi_{[t_j, \infty)}(t)+\eta(t)\\
u(0)=u_0.\\
\end{cases}
\end{equation}
Since the operator $A$ generates a strongly continuous semigroup $T$, by Proposition \ref{pp1}, we can find real numbers $M$ and $a$ satisfying $||T(t)||\le Me^{at}$ for all $t\ge 0$. We will use these numbers to estimate the accuracy of our recovery algorithm.

We acquire the following set of measurements:
\begin{equation}\label{meas}
    \begin{aligned}
    \mathfrak m_n(\frac{g}{\beta})=&\langle u(n\beta), \frac{g}{\beta}\rangle+\nu(n\beta, \frac{g}{\beta}),\\
    \mathfrak m_n(\frac{T^*(\beta)g}{\beta})=&\langle u(n\beta), \frac{T^*(\beta)g}{\beta}\rangle+\nu(n\beta, \frac{T^*(\beta)g}{\beta}),\quad g\in\widetilde G, n\in \N,
    \end{aligned}
\end{equation}
where $\beta$ is the time sampling step, and $\nu$ represents an additive noise (see Assumption \ref{as5}).
%$\nu(n\beta, \frac{g}{\beta}),\ \nu(n\beta, \frac{T^*(\beta)g}{\beta})$ represent additive noise that is assumed to be bounded according to Assumption \ref{as5}:
%
%$$\sup\limits_{n,g}| \nu(n\beta, \frac{g}{\beta})|\le \sigma, \quad \sup\limits_{n,g}| \nu(n\beta, \frac{T^*(\beta)g}{\beta})|\le \sigma,$$
%for some $\sigma\ge0$. 
The first of the pair of measurements in \eqref{meas} serves to assess the current state of the system whereas the second one will be used as a predictor of the measurement at $t=(n+1)\beta$.

To explain our idea of the recovery algorithm, we first present what happens in the ideal case when the measurements are noiseless ($\nu\equiv 0$) and there is no background source ($\eta \equiv 0$). For the convenience of exposition, we define $f_n\in\HH$ and $\tau_n\in [n\beta, (n+1)\beta)$ for each $n$ as follows
\begin{equation}\label{f_n}
\left\{
\begin{array}{lll}
f_n=h_j, &\tau_n=t_j, & \mbox{ if the } j\mbox{-th burst occurred in } [n\beta, (n+1)\beta);\\
f_n=0, &\tau_n=n\beta, & \mbox{ if no burst occurred in } [n\beta, (n+1)\beta).
\end{array}
\right.
\end{equation}
There is no ambiguity in the above definition due to Assumption \ref{as6}.

To reveal the predictive nature of the second measurement in \eqref{meas}, we first consider the difference
\begin{equation}\label{Fn}
F_n=\mathfrak m_{n+1}(\frac{g}{\beta})-\mathfrak m_n(\frac{T^*(\beta)g}{\beta}).
\end{equation}
In the ideal case, utilizing \eqref{solburst} we get
\[
\begin{array}{ll}
F_n&=\mathfrak m_{n+1}(\frac{g}{\beta})-\mathfrak m_n(\frac{T^*(\beta)g}{\beta})\\
&=\langle T((n+1)\beta)u_0,\frac{g}{\beta}\rangle+\sum\limits_{\tau_i<(n+1)\beta}\int_{\tau_i}^{(n+1)\beta}\langle T((n+1)\beta-s)f_i e^{\rho(\tau_i-s)}, \frac{g}{\beta}\rangle ds
\\
&\quad-\langle T(n\beta)u_0,\frac{T^*(\beta)g}{\beta}\rangle-
\sum\limits_{\tau_i<n\beta}\int_{\tau_i}^{n\beta}\langle T(n\beta-s)f_i e^{\rho(\tau_i-s)}, \frac{T^*(\beta)g}{\beta}\rangle ds \\ 

&=\int_{\tau_n}^{(n+1)\beta}\langle T((n+1)\beta-s)f_n e^{d(\tau_n-s)}, \frac{g}{\beta}\rangle ds
\\
&\quad+
\sum\limits_{\tau_i<n\beta}\int_{\tau_i}^{(n+1)\beta} \langle T((n+1)\beta-s)f_i e^{\rho(\tau_i-s)}, \frac{g}{\beta}\rangle ds
\\
&\quad-
\sum\limits_{\tau_i<n\beta}\int_{\tau_i}^{n\beta}\langle T((n+1)\beta-s)f_i e^{\rho(\tau_i-s)}, \frac{g}{\beta}\rangle ds 
\\

&=\int_{\tau_n}^{(n+1)\beta}\langle T((n+1)\beta-s)f_n e^{\rho(\tau_n-s)}, \frac{g}{\beta}\rangle ds
\\
&\quad+
\sum\limits_{\tau_i<n\beta}\int_{n\beta}^{(n+1)\beta} \langle T((n+1)\beta-s)f_i e^{\rho(\tau_i-s)}, \frac{g}{\beta}\rangle ds.
\\

&=\int_{\tau_n}^{(n+1)\beta}\langle T((n+1)\beta-s)f_n e^{\rho(\tau_n-s)}, \frac{g}{\beta}\rangle ds
\\
&\quad+
\sum\limits_{\tau_i<n\beta}e^{-\rho n\beta}\int_{0}^{\beta} \langle T(\beta-s)f_i e^{\rho(\tau_i-s)}, \frac{g}{\beta}\rangle ds.
\end{array}
\]
\begin{remark}
In the expression for $F_n$ above, if no burst occurred in $[n\beta, (n+1)\beta)$ (i.e. $f_n=0$), then $\int_{\tau_n}^{(n+1)\beta}\langle T((n+1)\beta-s)f_n e^{\rho(\tau_n-s)}, \frac{g}{\beta}\rangle ds=0.$ In addition, the term $\sum\limits_{\tau_i<n\beta}e^{-\rho n\beta}\int_0^{\beta}\langle T(\beta-s)f_i e^{\rho(\tau_i-s)}, \frac{g}{\beta}\rangle ds$ represents the effect of the bursts that had occurred before $n\beta$. 
\end{remark}

Secondly, we calculate the difference  $\Delta_n=e^{\rho\beta}F_{n+1}-F_n$, 
which involves the measurements in two consecutive intervals $[n\beta, (n+1)\beta)$ and $[(n+1)\beta, (n+2)\beta)$:
\[
\begin{array}{ll}
\Delta_n&=e^{\rho\beta}F_{n+1}-F_n\\
&=e^{\rho\beta}\int_{\tau_{n+1}}^{(n+2)\beta}\langle T((n+2)\beta-s)f_{n+1} e^{\rho(\tau_{n+1}-s)}, \frac{g}{\beta}\rangle ds+\sum\limits_{\tau_i<(n+1)\beta}e^{-\rho n\beta}\int_0^{\beta}\langle T(\beta-s)f_i e^{\rho(\tau_i-s)}, \frac{g}{\beta}\rangle ds\\

&\quad-\int_{\tau_n}^{(n+1)\beta}\langle T((n+1)\beta-s)f_n e^{\rho(\tau_n-s)}, \frac{g}{\beta}\rangle ds-\sum\limits_{\tau_i<n\beta}e^{-\rho n\beta}\int_0^{\beta}\langle T(\beta-s)f_i e^{\rho(\tau_i-s)}, \frac{g}{\beta}\rangle ds\\

&=e^{\rho\beta}\int_{\tau_{n+1}}^{(n+2)\beta}\langle T((n+2)\beta-s)f_{n+1} e^{\rho(\tau_{n+1}-s)}, \frac{g}{\beta}\rangle ds\\
&\quad-\int_{\tau_n}^{(n+1)\beta}\langle T((n+1)\beta-s)f_n e^{\rho(\tau_n-s)}, \frac{g}{\beta}\rangle ds+\sum\limits_{n\beta\le \tau_i<(n+1)\beta}e^{-\rho n\beta}\int_0^{\beta}\langle T(\beta-s)f_i e^{\rho(\tau_i-s)}, \frac{g}{\beta}\rangle ds\\

&=e^{\rho\beta}\int_{\tau_{n+1}}^{(n+2)\beta}\langle T((n+2)\beta-s)f_{n+1} e^{\rho(\tau_{n+1}-s)}, \frac{g}{\beta}\rangle ds\\
&\quad-\int_{\tau_n}^{(n+1)\beta}\langle T((n+1)\beta-s)f_n e^{\rho(\tau_n-s)}, \frac{g}{\beta}\rangle ds+e^{-\rho n\beta}\int_0^{\beta}\langle T(\beta-s)f_n e^{\rho(\tau_n-s)}, \frac{g}{\beta}\rangle ds\\

&=e^{\rho\beta}\int_{\tau_{n+1}}^{(n+2)\beta}\langle T((n+2)\beta-s)f_{n+1} e^{\rho(\tau_{n+1}-s)}, \frac{g}{\beta}\rangle ds\\
&\quad+\int_{n\beta}^{\tau_n}\langle T((n+1)\beta-s)f_n e^{\rho(\tau_n-s)}, \frac{g}{\beta}\rangle ds.
\end{array}
\]

The above calculation leads us to the key observation
\begin{equation}
    \label{keyobs}
    \Delta_n=0\quad\mbox{in the ideal case if no burst occurred in} [n\beta, (n+2)\beta).
\end{equation}
In case the $j$-th burst did occur in the interval $[(n+1)\beta, (n+2)\beta)$, we use the following calculation to estimate the inner products  $\langle h_j,g\rangle$. In view of Assumption \ref{as6}, we have $f_{n-1}=f_n=f_{n+2}=0$ and,
for $\beta$ sufficiently small, we get

\[
\begin{array}{ll}
&\quad e^{3\rho\beta}F_{n+2}-F_{n-1}\\
&=e^{3\rho\beta}\int_{\tau_{n+2}}^{(n+3)\beta}\langle T((n+3)\beta-s)f_{n+2} e^{\rho(\tau_{n+2}-s)}, \frac{g}{\beta}\rangle ds+\sum\limits_{\tau_i<(n+2)\beta}e^{-\rho (n-1)\beta}\int_0^{\beta}\langle T(\beta-s)f_i e^{\rho(\tau_i-s)}, \frac{g}{\beta}\rangle ds\\

&\quad-\int_{\tau_{n-1}}^{n\beta}\langle T(n\beta-s)f_{n-1} e^{\rho(\tau_{n-1}-s)}, \frac{g}{\beta}\rangle ds-\sum\limits_{\tau_i<(n-1)\beta}e^{-\rho (n-1)\beta}\int_0^{\beta}\langle T(\beta-s)f_i e^{\rho(\tau_i-s)}, \frac{g}{\beta}\rangle ds\\

&=\int_0^{\beta}\langle T(\beta-s)f_{n+1} e^{\rho(\tau_{n+1}-(n-1)\beta-s)}, \frac{g}{\beta}\rangle ds\\

&=\int_0^{\beta}\langle T(\beta-s)h_j e^{\rho(t_j-(n-1)\beta-s)}, \frac{g}{\beta}\rangle ds
\approx \langle h_j,g\rangle,
\end{array}
\]
where the last assertion is (essentially) yielded by the following lemma.

\begin{lemma}\label{lem_inpr}
Assume that $t_j\in [(n+1)\beta, (n+2)\beta)$ and 
\begin{equation}\label{v_1}
v_{k}(h_j,g,\beta)=\left|\int_0^{\beta}\langle T(\beta-s)h_je^{\rho(t_j-(n-k)\beta-s)}, \frac{g}{\beta}\rangle ds-\langle h_j, g\rangle\right|,\ k=0,1. 
\end{equation}
Then
\begin{equation}\label{estv1}
    v_k(h_j,g,\beta) \le \|g\|\left(M\|h_j\|(e^{(k+2)\rho\beta}-1)\mathbf{e}(a\beta)+\sup\limits_{s\in[0,\beta]}\|T(s)h_j-h_j\|\right),
\end{equation}
where $M$ and $a$ are as in Proposition \ref{pp1} and 
\begin{equation}\label{bfe}
\mathbf{e}(t)=
\left\{
\begin{array}{lll}
\frac{e^{t}-1}{t}, & t\neq 0;\\
1, & t=0.
\end{array}
\right.
\end{equation}
In particular, $v_k(h_j,g,\beta) \rightarrow 0$ as $\beta \rightarrow 0$.
\end{lemma}

\begin{proof}
Observe that 
\begin{equation}\label{error11}
\begin{array}{ll}
v_k(h_j,g,\beta)&=|\int_0^{\beta}\langle T(\beta-s)h_je^{\rho(t_j-(n-k)\beta-s)}, \frac{g}{\beta}\rangle ds-\langle h_j, g\rangle|\\

&=|\int_0^{\beta}\langle T(\beta-s)h_je^{\rho(t_j-(n-k)\beta-s)}-h_j, \frac{g}{\beta}\rangle ds|\\

&\le\int_0^{\beta}\|T(\beta-s)h_je^{\rho(t_j-(n-k)\beta-s)}-h_j\|\frac{\|g\|}{\beta} ds\\

&\le\frac{\|g\|}{\beta}\int_0^{\beta} \|T(\beta-s)h_je^{\rho(t_j-(n-k)\beta-s)}-T(\beta-s)h_j\|ds\\

&\quad+\frac{\|g\|}{\beta}\int_0^{\beta} \|T(\beta-s)h_j-h_j\|ds\\
&=I_1+I_2.
\end{array}
\end{equation}
Using Proposition \ref{pp1}, i.e.~the inequality $\|T(t)\|\le Me^{at}$, we get
\begin{equation}\label{error12}
\begin{array}{ll}
I_1&=\frac{\|g\|}{\beta}\int_0^{\beta} \|T(\beta-s)h_je^{\rho(t_j-(n-k)\beta-s)}-T(\beta-s)h_j\|ds\\
&\le\frac{\|g\|}{\beta}\int_0^{\beta}\|T(\beta-s)\|\|h_j\|(e^{(k+2)\rho\beta}-1)ds\\

&\le\|g\|M\|h_j\|(e^{(k+2)\rho\beta}-1)\frac{1}{\beta}\int_0^{\beta}e^{a(\beta-s)}ds\\
&\le\|g\|\|h_j\|M(e^{(k+2)\rho\beta}-1)\mathbf{e}(a\beta)
\end{array}
\end{equation}
and
\begin{equation}\label{error13}
\begin{array}{ll}
I_2&=\frac{\|g\|}{\beta}\int_0^{\beta} \|T(\beta-s)h_j-h_j\|ds\\
&=\frac{\|g\|}{\beta}\int_0^{\beta}\|T(s)h_j-h_j\|ds\\
&\le \|g\| \sup\limits_{s\in[0,\beta]}\|T(s)h_j-h_j\|.
\end{array}
\end{equation}
Estimate \eqref{estv1} immediately follows from \eqref{error11}, \eqref{error12}, and \eqref{error13}. We get $\lim\limits_{\beta\to 0}v_k(h_j,g,\beta) = 0$ since
$\lim\limits_{\beta\to 0} (e^{(k+2)\rho\beta}-1)\mathbf e(a\beta) = 0$ and $\lim\limits_{\beta\to 0} \sup\limits_{s\in[0,\beta]}\|T(s)h_j-h_j\|=0$ due to the strong continuity of the semigroup $T$. 
\end{proof}

%More precisely, we  use measurements on four intervals of length $\beta$ to estimate a burst and its value.
Equipped with the above observations, we are naturally led to Algorithm 1 below. The algorithm turns out to be robust both with respect to the considered additive measurement noise and introduction of a background source as described in the following Theorem \ref{thm1}.

\bigskip

\begin{tabular}{rp{13cm}}%\label{alg1}
\toprule
\multicolumn{2}{p{13cm}}{\textbf{Algorithm 1.} Pseudo-code for approximating the time and shape of a possible burst with prescribed exponential decay}\\
\midrule
1:& \textbf{Input:} Measurements: $\mathfrak m_{n}(\frac{g}{\beta})$, $\mathfrak m_n(\frac{T^*(\beta)g}{\beta})$ and threshold: $Q(g,\beta)$, $g\in\widetilde G$\\
2:& Compute $F_i=\mathfrak m_{i+1}(\frac{g}{\beta})-\mathfrak m_i(\frac{T^*(\beta)g}{\beta})$\\
3:& Compute $e^{\rho\beta} F_{i+1}-F_i$\\
4:& \textbf{For} $g\in\widetilde G$ \textbf{do}\\
5:& \quad i=1\\
6:& \quad\textbf{while} $i\beta<T$\\
7:& \quad\quad\textbf{if} $e^{\rho\beta} F_{i+1}-F_i> Q(g,\beta)$\textbf{then}\\
8:& \quad\quad\quad$\mathfrak f(g):=e^{3\rho\beta} F_{i+2}-F_{i-1}$\\
9:& \quad\quad\quad$  \estime :=(i+1)\beta$\\
10:& \quad\quad\quad$i=i+3$\\
11:& \quad\quad\textbf{else}\\
12:& \quad\quad\quad\textbf{if} $e^{\rho\beta}F_{i+2}-F_{i+1}>Q(g,\beta)$ \textbf{then}\\
13:& \quad\quad\quad\quad$\mathfrak f(g):=e^{3\rho\beta} F_{i+3}-F_{i}$\\
14:& \quad\quad\quad\quad$  \estime :=(i+2)\beta$\\
15:& \quad\quad\quad\quad$i=i+3$\\
16:& \quad\quad\quad\textbf{else}\\
17:& \quad\quad\quad\quad$i=i+1$\\
18:& \textbf{Output:} $  \estime $ and $\mathfrak f(g)$ for all $g\in\widetilde G$.\\
\bottomrule
\end{tabular}
\bigskip

\begin{theorem}\label{thm1}
Under Assumptions \ref{as2},\ref{as4},\ref{as5} and \ref{as6}, and
$M, a$ as in Proposition \ref{pp1}, let % $C=Me^{a\beta}$ and
\begin{equation}\label{thresh}
    Q(g,\beta)=e^{(\rho+a)\beta}ML\|g\|\beta+e^{a\beta}(e^{\beta}-1)MK\|g\|+4e^{\rho\beta}\sigma
\end{equation}
be the threshold in Algorithm 1. Let also $\mathfrak t_j$ and $\mathfrak f_j(g)$ be the outputs of Algorithm 1. Then $|\estime_j-t_j|\le \beta$ and
%
%we are able to find out the bursts $\{h_j\}$. For the burst time $t_j$, we have $|\estime_j-t_j|\le \beta$. For the burst shape $\langle h_{j}, g\rangle$, it can be well approximated by $\mathfrak f_j(g)$ for all $g$. In particular, we have
\begin{equation}\label{inpest}
\begin{split}
|\mathfrak f_j(g)&-\langle h_{j}, g\rangle|\\
&\le 3e^{(3\rho+a)\beta}ML\|g\|\beta+e^{a\beta}(e^{3\rho\beta}-1)MK\|g\|+4e^{3\rho\beta}\sigma\\ &\quad+2e^{\rho\beta}Q(g,\beta)+\max\{v_0(h_j,g,\beta), v_1(h_j,g,\beta)\},
\end{split}
\end{equation}
where $v_k$, $k=0,1$, are given by \eqref{estv1}. In particular, for sufficiently small $\beta>0$, one has
$|\mathfrak f_j(g)-\langle h_{j}, g\rangle| \le 13\sigma$ as long as $\sigma > 0$.

\end{theorem}

\begin{proof}
Adjusting the computations in the ideal case to account for the noise and the background source, we get
%Computing $F_n=\mathfrak m_{n+1}(\frac{g}{\beta})-\mathfrak m_n(\frac{T^*(\beta)g}{\beta})$ directly, we get
\[
\begin{array}{ll}
F_n&=\mathfrak m_{n+1}(\frac{g}{\beta})-\mathfrak m_n(\frac{T^*(\beta)g}{\beta})\\
%&=\sum\limits_{\tau_i<(n+1)\beta}\int_{\tau_i}^{(n+1)\beta}\langle T((n+1)\beta-s)f_i e^{\rho(\tau_i-s)}, \frac{g}{\beta}\rangle ds+\int_0^{(n+1)\beta}\langle T((n+1)\beta-s)\eta(s),\frac{g}{\beta}\rangle ds\\

%&\quad-\sum\limits_{\tau_i<n\beta}\int_{\tau_i}^{n\beta}\langle T(n\beta-s)f_i e^{\rho(\tau_i-s)}, \frac{T^*(\beta)g}{\beta}\rangle ds-\int_0^{n\beta}\langle T(n\beta-s)\eta(s),\frac{T^*(\beta)g}{\beta}\rangle ds\\

%&\quad+\nu((n+1)\beta, \frac{g}{\beta})-\nu(n\beta, \frac{T^*(\beta)g}{\beta})\\

%&=\int_{\tau_n}^{(n+1)\beta}\langle T((n+1)\beta-s)f_n e^{\rho(\tau_n-s)}, \frac{g}{\beta}\rangle ds+\sum\limits_{\tau_i<n\beta}\int_{n\beta}^{(n+1)\beta}\langle T((n+1)\beta-s)f_i e^{\rho(\tau_i-s)}, \frac{g}{\beta}\rangle ds\\

%&\quad+\int_{n\beta}^{(n+1)\beta}\langle T((n+1)\beta-s)\eta(s),\frac{g}{\beta}\rangle ds+\nu((n+1)\beta, \frac{g}{\beta})-\nu(n\beta, \frac{T^*(\beta)g}{\beta})\\

&=\int_{\tau_n}^{(n+1)\beta}\langle T((n+1)\beta-s)f_n e^{\rho(\tau_n-s)}, \frac{g}{\beta}\rangle ds+\sum\limits_{\tau_i<n\beta}\int_0^{\beta}\langle T(\beta-s)f_i e^{\rho(\tau_i-n\beta-s)}, \frac{g}{\beta}\rangle ds\\

&\quad+\int_0^{\beta}\langle T(\beta-s)\eta(n\beta+s), \frac{g}{\beta}\rangle ds+\nu((n+1)\beta, \frac{g}{\beta})-\nu(n\beta, \frac{T^*(\beta)g}{\beta}).\\
\end{array}
\]
The difference $\Delta_n=e^{\rho\beta}F_{n+1}-F_n$ not only allows us to detect the burst in the ideal case (where it kills the effect of the past bursts according to \eqref{keyobs}), but also, as we shall see presently, mitigates the effect of the background source. Once again, adjusting the previous computations for noise and background source, we get
%
%
%In order to eliminate the effect of the past bursts and minimize the effect of background source, we consider the difference 
\[
\begin{array}{ll}
\Delta_n&=e^{\rho\beta}F_{n+1}-F_n\\

&=\int_{\tau_{n+1}}^{(n+2)\beta}\langle T((n+2)\beta-s)f_{n+1} e^{\rho(\tau_{n+1}+\beta-s)}, \frac{g}{\beta}\rangle ds+\int_{n\beta}^{\tau_n}\langle T((n+1)\beta-s)f_n e^{\rho(\tau_n-s)}, \frac{g}{\beta}\rangle ds\\

&\quad+e^{\rho\beta}\int_0^{\beta}\langle T(\beta-s)\eta((n+1)\beta+s), \frac{g}{\beta}\rangle ds-\int_0^{\beta}\langle T(\beta-s)\eta(n\beta+s), \frac{g}{\beta}\rangle ds+\alpha_n,
\end{array}
\]
where $\alpha_n=e^{\rho\beta}\nu((n+2)\beta, \frac{g}{\beta})-e^{\rho\beta}\nu((n+1)\beta, \frac{T^*(\beta)g}{\beta})-\nu((n+1)\beta, \frac{g}{\beta})+\nu(n\beta, \frac{T^*(\beta)g}{\beta}).$

We remark that Assumption \ref {as6} ($ t_{j+1}-t_j>4\beta)$ implies that, at  most one of the terms $f_n$, $f_{n+1}$ is non-zero in the expression  of $e^{\rho\beta}F_{n+1}-F_n$ above. 

Firstly,  we  prove  that if no burst occurred in $[n\beta, (n+2)\beta)$ (i.e. $f_n=f_{n+1}=0$), then $|\Delta_n|$ is below our chosen threshold \eqref{thresh}. %In this case, using Assumptions \ref{as4} and \ref{as5}, we estimate $|\Delta_n|$ as follows
The proof is achieved via the following computations that make use of Assumptions \ref{as4} and \ref{as5}:
\begin{equation} \label{Thresholderr} 
\begin{array}{ll}
\lvert\Delta_n\rvert&=|e^{\rho\beta}F_{n+1}-F_n|\\
&=|e^{\rho\beta}\int_0^{\beta}\langle T(\beta-s)\eta((n+1)\beta+s), \frac{g}{\beta}\rangle ds-\int_0^{\beta}\langle T(\beta-s)\eta(n\beta+s), \frac{g}{\beta}\rangle ds+\alpha_n|\\

&\le |e^{\rho\beta}\int_0^{\beta}\langle T(\beta-s)\eta((n+1)\beta+s), \frac{g}{\beta}\rangle ds-e^{\rho\beta}\int_0^{\beta}\langle T(\beta-s)\eta(n\beta+s), \frac{g}{\beta}\rangle ds|\\

&\quad+|e^{\rho\beta}\int_0^{\beta}\langle T(\beta-s)\eta(n\beta+s), \frac{g}{\beta}\rangle ds-\int_0^{\beta}\langle T(\beta-s)\eta(n\beta+s), \frac{g}{\beta}\rangle ds|+|\alpha_n|\\

&\le e^{\rho\beta}\int_0^{\beta}\|\eta((n+1)\beta+s)-\eta(n\beta+s)\|\|T^*(\beta-s)\frac{g}{\beta}\|ds\\
&\quad+(e^{\rho\beta}-1)\int_0^{\beta}\|\eta(n\beta+s)\|\|T^*(\beta-s)\frac{g}{\beta}\|ds+4e^{\rho\beta}\sigma\\

&\le e^{\rho\beta}CL\|g\|\beta+(e^{\rho\beta}-1)CK\|g\|+4e^{\rho\beta}\sigma=Q(g,\beta)\\
\end{array}
\end{equation}
where $C=Me^{a\beta}$ so that $\|T^*(\beta-s)\|\le C.$

Secondly, assume that the $j$-th burst with the shape $h_j$ occurred in the interval $[(n+1)\beta, (n+2)\beta)$. To analyze this situation, we will look at two cases: 1)  the $j$-th burst is detected by our algorithm; and 2) the $j$-th burst is not detected. For Case 1, the burst is detected if and only if $\lvert \Delta_n\rvert>Q(g,\beta)$ or $\lvert \Delta_{n+1}\rvert>Q(g,\beta)$ and we need to prove \eqref{inpest}. For Case 2, when the $j$-th burst  is not detected, we will show that $\langle h_j,g\rangle$ is small, i.e. \eqref{inpest} holds with $\mathfrak f_j(g)=0$.\\

\medskip\noindent
{\it Case 1.} The $j$-th burst is detected in $[(n+1)\beta, (n+2)\beta)$. 
\medskip

Assume that $\lvert \Delta_{n}\rvert>Q(g,\beta)$. Then $\tau_{n+1}=t_j$, $f_{n+1}=h_j$, $f_{n-1}=f_{n}=f_{n+2}=0$ and Algorithm 1 returns  $\estime_j=(n+1)\beta$ and $\mathfrak f_j(g) = e^{3\rho\beta}F_{n+2}-F_{n-1}$. We need to establish \eqref{inpest}, i.e.
show that  $\langle h_j,g\rangle\approx e^{3\rho\beta}F_{n+2}-F_{n-1}$ for small $\beta$. We get
 
 \[
\begin{split}
 &\quad e^{3\rho\beta}F_{n+2}-F_{n-1}\\
 &=\int_0^{\beta}\langle T(\beta-s)h_j e^{\rho(t_j-(n-1)\beta-s)}, \frac{g}{\beta}\rangle ds+e^{3\rho\beta}\int_0^{\beta}\langle T(\beta-s)\eta((n+2)\beta+s), \frac{g}{\beta}\rangle ds\\
&\quad-\int_0^{\beta}\langle T(\beta-s)\eta((n-1)\beta+s), \frac{g}{\beta}\rangle ds+\alpha_{n-1}^{'}
\end{split}
\]
where $\alpha_{n-1}^{'}=e^{3\rho\beta}\nu((n+3)\beta, \frac{g}{\beta})-e^{3\rho\beta}\nu((n+2)\beta, \frac{T^*(\beta)g}{\beta})-\nu(n\beta, \frac{g}{\beta})+\nu((n-1)\beta, \frac{T^*(\beta)g}{\beta}).$ Therefore,
 \begin{equation*}
\begin{array}{ll}
&\quad|\mathfrak f_j(g)-\langle h_j, g\rangle|\\
&=|e^{3\rho\beta}F_{n+2}-F_{n-1}-\langle h_j, g\rangle|\\
&\le|\int_0^{\beta}\langle T(\beta-s)h_je^{\rho(t_j-(n-1)\beta-s)}, \frac{g}{\beta}\rangle ds-\langle h_j, g\rangle|\\

&\quad+|e^{3\rho\beta}\int_0^{\beta}\langle T(\beta-s)\eta((n+2)\beta+s), \frac{g}{\beta}\rangle ds-\int_0^{\beta}\langle T(\beta-s)\eta((n-1)\beta+s) \frac{g}{\beta}\rangle ds|+4e^{3\rho\beta}\sigma\\

&\le v_1(h_j,g,\beta)+3e^{3\rho\beta}CL\|g\|\beta+(e^{3\rho\beta}-1)CK\|g\|+4e^{3\rho\beta}\sigma,
\end{array}
\end{equation*}
 where $C=Me^{a\beta}$ and $v_1(h_j,g,\beta)$ is given by \eqref{estv1}. Thus, estimate \eqref{inpest} is established for the case when $\lvert \Delta_{n}\rvert>Q(g,\beta)$.
 
%\begin{remark}
Assume now that $\lvert \Delta_{n}\rvert\le Q(g,\beta)$ and $\lvert \Delta_{n+1}\rvert>Q(g,\beta)$. In this case, Algorithm 1 returns $\estime_j=(n+2)\beta$ and $\mathfrak f_j(g)=e^{3\rho\beta}F_{n+3}-F_{n}$. In particular, 
 \begin{equation*}
\begin{array}{ll}
&\quad|\mathfrak f_j(g)-\langle h_j, g\rangle|\\
&\le|\int_0^{\beta}\langle T(\beta-s)h_je^{\rho(t_j-n\beta-s)}, \frac{g}{\beta}\rangle ds-\langle h_j, g\rangle|\\
&\quad+3e^{3\rho\beta}CL\|g\|\beta+(e^{3\rho\beta}-1)CK\|g\|+4e^{3\rho\beta}\sigma,\\
\end{array}
\end{equation*}
 where $|\int_0^{\beta}\langle T(\beta-s)h_je^{\rho(t_j-n\beta-s)}, \frac{g}{\beta}\rangle ds-\langle h_j, g\rangle|=v_0(h_j,g,\beta)$ is given by \eqref{estv1}. Thus, estimate \eqref{inpest} holds when $\lvert \Delta_{n+1}\rvert>Q(g,\beta)$ as well, and Case 1 is covered.
 
%\end{remark}

\bigskip\noindent 
{\it Case 2.} The $j$-th burst is in $[(n+1)\beta, (n+2)\beta)$, but $\langle h_j, g\rangle$ is too small to be detected. 
\medskip

We need to show that \eqref{inpest} holds with $\mathfrak f_j(g)=0$. We have

\begin{equation} \label{diff2i}
\begin{array}{ll}
&\quad e^{2\rho\beta}F_{n+2}-F_n\\
&=e^{\rho\beta} \Delta_{n+1}+ \Delta_{n}\\
&=\int_0^{\beta}\langle T(\beta-s)h_j e^{\rho(t_j-n\beta-s)}, \frac{g}{\beta}\rangle ds+e^{2\rho\beta}\int_0^{\beta}\langle T(\beta-s)\eta((n+2)\beta+s), \frac{g}{\beta}\rangle ds\\

&\quad-\int_0^{\beta}\langle T(\beta-s)\eta(n\beta+s), \frac{g}{\beta}\rangle ds+\alpha_n^{''}\\
\end{array}
\end{equation}
where $\alpha_n^{''}=e^{2\rho\beta}\nu((n+3)\beta, \frac{g}{\beta})-e^{2\rho\beta}\nu((n+2)\beta, \frac{T^*(\beta)g}{\beta})-\nu((n+1)\beta, \frac{g}{\beta})+\nu(n\beta, \frac{T^*(\beta)g}{\beta}).$ Using \eqref{diff2i} to estimate $\langle h_j, g\rangle$, we get

\begin{equation*}%\label{error15}
\begin{array}{ll}
&\quad|\langle h_j, g\rangle|\\
&\le|-\int_0^{\beta}\langle T(\beta-s)h_je^{\rho(t_j-n\beta-s)}, \frac{g}{\beta}\rangle ds|+|\int_0^{\beta}\langle T(\beta-s)h_je^{\rho(t_j-n\beta-s)}, \frac{g}{\beta}\rangle ds-\langle h_j, g\rangle|\\

&=|-\int_0^{\beta}\langle T(\beta-s)h_je^{\rho(t_j-n\beta-s)}, \frac{g}{\beta}\rangle ds+(e^{2\rho\beta}F_{n+2}-F_n)-(e^{2\rho\beta}F_{n+2}-F_n)|+v_0(h_j,g,\beta)\\
    
&\le|-\int_0^{\beta}\langle T(\beta-s)h_je^{\rho(t_j-n\beta-s)}, \frac{g}{\beta}\rangle ds+(e^{2\rho\beta}F_{n+2}-F_n)|+|e^{2\rho\beta}F_{n+2}-F_n|+v_0(h_j,g,\beta)\\
&\le|e^{2\rho\beta}\int_0^{\beta}\langle T(\beta-s)\eta((n+2)\beta+s), \frac{g}{\beta}\rangle ds-\int_0^{\beta}\langle T(\beta-s)\eta(n\beta+s), \frac{g}{\beta}\rangle ds+\alpha_n''|\\

&\quad+e^{\rho\beta}|e^{\rho\beta}F_{n+2}-F_{n+1}|+|e^{\rho\beta}F_{n+1}-F_n|+v_0(h_j,g,\beta)\\

&\le 2e^{2\rho\beta}CL\|g\|\beta+(e^{2\rho\beta}-1)CK\|g\|+4e^{2\rho\beta}\sigma+2e^{\rho\beta}Q(g,\beta)+v_0(h_j,g,\beta),\\
\end{array}
\end{equation*}
where we have used the fact that $|e^{\rho\beta}F_{n+2}-F_{n+1}|\le Q(g,\beta)$, $|e^{\rho\beta}F_{n+1}-F_{n}|\le Q(g,\beta)$, and  estimated the term
$|e^{2\rho\beta}\int_0^{\beta}\langle T(\beta-s)\eta((n+2)\beta+s), \frac{g}{\beta}\rangle ds-\int_0^{\beta}\langle T(\beta-s)\eta(n\beta+s), \frac{g}{\beta}\rangle ds+\alpha_n''|$ in a similar way as \eqref {Thresholderr}.
The above estimates establish \eqref{inpest} in Case 2, and the theorem is proved. 
\end{proof}

\subsection{Model with general decay function}\label{model2}

In this section, we consider the same dynamic system, but we discuss a more general situation. Here the decay function $\phi(t)$ does not have a concrete formula, but its decay velocity is restricted. The model is as follows
\begin{equation}
\left\{
\begin{array}{l}
\dot{u}(t)=Au(t)+\sum\limits_j^N h_j\phi(t-t_j)\chi_{[t_j, \infty)}(t)+\eta,\\
u(0)=u_0\\
\end{array}
\right.
\end{equation}
where the function $\phi(t)$ is a continuous function on $[0,\infty)$ satisfying $\phi(0)=1$ and 
\begin{equation}\label{condition1}
    0<\phi(t)\le e^{-\rho t}
\end{equation}
for some $\rho>0$.

In this model, we continue to use the constants introduced in Proposition \ref{pp1} and Assumptions \ref{as2} to \ref{as6}, as well as $C=Me^{a\beta}$. We also assume that the bursts are uniformly bounded as mentioned in Assumption \ref{as3}.
The constant $D$ in Assumption \ref{as6}  that controls the time gap between the bursts ($t_{j+1}-t_j\ge D+4\beta$) is chosen in a way that 
\begin{equation}\label{ep}
\epsilon=\frac{2}{e^{\rho D}-1}CHR    
\end{equation}
is a small quantity. 
We shall also need the following modification of the technical lemma \ref{lem_inpr}.

\begin{lemma}\label{lem_inpr2}
Assume that $t_j\in [(n+1)\beta, (n+2)\beta)$ and 
\begin{equation}\label{v_2}
v_{k,i}(h_j,g,\beta)=\left|\int_0^{\beta}\langle T(\beta-s)h_je^{k\rho\beta}\phi((n+i)\beta+s-t_j), \frac{g}{\beta}\rangle ds-\langle h_j, g\rangle\right|,\ k,i=2,3. 
\end{equation}
Then
\begin{equation}\label{estvk}
    v_{k,i}(h_j,g,\beta) \le \|g\|\left(%M\|h_j\|(e^{2\rho\beta}-1)
    \|h_j\|M\max\limits_{s\in[(i-2)\beta,i\beta]}|e^{k\rho\beta}\phi(s)-1|
    \mathbf{e}(a\beta)+\sup\limits_{s\in[0,\beta]}\|T(s)h_j-h_j\|\right),
\end{equation}
where $\mathbf{e}$ is given by \eqref{bfe}.
In particular, $v_{k,i}(h_j,g,\beta) \rightarrow 0$ as $\beta \rightarrow 0$, $k,i=2,3$.
\end{lemma}

\begin{proof}
Similarly to \eqref{error11}, we separate each $v_{k,i}(h_j,g,\beta)$, $k,i=2,3$, into two parts:

\[
\begin{array}{ll}
&\quad|\int_0^{\beta}\langle T(\beta-s)h_je^{k\rho\beta}\phi((n+i)\beta+s-t_j), \frac{g}{\beta}\rangle ds-\langle h_j, g\rangle|\\

&\le\frac{\|g\|}{\beta}\int_0^{\beta} \|T(\beta-s)h_je^{k\rho\beta}\phi((n+i)\beta+s-t_j)-T(\beta-s)h_j\|ds\\

&\quad+\frac{\|g\|}{\beta}\int_0^{\beta} \|T(\beta-s)h_j-h_j\|ds\\
&=I_1+I_2.
\end{array}
\]

Estimate for $I_2$ is still given by \eqref{error13}.
For $I_1$, by $\|T(t)\|\le Me^{at}$, we have
\[
\begin{array}{ll}
I_1&=\frac{\|g\|}{\beta}\int_0^{\beta} \|T(\beta-s)h_je^{k\rho\beta}\phi((n+i)\beta+s-t_{j})-T(\beta-s)h_j\|ds\\
&\le\frac{\|g\|}{\beta}\int_0^{\beta}\|T(\beta-s)\|\|h_j\||e^{k\rho\beta}\phi((n+i)\beta+s-t_{j})-1|ds\\

&=\|g\|\|h_j\|M\max\limits_{s\in[(i-2)\beta,i\beta]}|e^{k\rho\beta}\phi(s)-1|\mathbf{e}(a\beta).
\end{array}
\]
By the assumption on $\phi$, $I_1\rightarrow 0$ as $\beta\rightarrow 0$.
\end{proof}

\bigskip
\begin{tabular}{rp{13cm}}
\toprule
\multicolumn{2}{p{13cm}}{\textbf{Algorithm 2} Pseudo-code for approximating the time and shape of a possible burst with varying decay}\\
\midrule
1:&\textbf{Input:} Measurements: $\mathfrak m_{n}(\frac{g}{\beta})$, $\mathfrak m_n(\frac{T^*(\beta)g}{\beta})$; threshold: $Q_1(g,\beta)$, for $g\in\widetilde G$; a parameter $D>0$\\
2:&Compute $F_i=\mathfrak m_{i+1}(\frac{g}{\beta})-\mathfrak m_i(\frac{T^*(\beta)g}{\beta})$\\
3:&Compute $e^{\rho\beta} F_{i+1}-F_i$\\
4:&\textbf{For} $g\in\widetilde G$ \textbf{do}\\
5:& i=1\\
6:&\quad\textbf{while} $i\beta<T$ \textbf{do}\\
7:&\quad\quad\textbf{if} $e^{\rho\beta} F_{i+1}-F_i>Q_1(g,\beta)$ \textbf{then}\\
8:&\quad\quad\quad$\mathfrak f(g):=e^{3\rho\beta} F_{i+2}-F_{i-1}$\\
9:&\quad\quad\quad$  \estime: =(i+1)\beta$\\
10:&\quad\quad\quad$i=i+3+\lfloor\frac{D}{\beta}\rfloor$\\ 
11:&\quad\quad\textbf{else}\\
12:&\quad\quad\quad\textbf{if} $e^{\rho\beta} F_{i+2}-F_{i+1}>Q_1(g,\beta)$ \textbf{then}\\
13:&\quad\quad\quad\quad$\mathfrak f(g):=e^{3\rho\beta} F_{i+3}-F_{i}$\\
14:&\quad\quad\quad\quad$  \estime: =(i+2)\beta$\\
15:&\quad\quad\quad\quad$i=i+3+\lfloor\frac{D}{\beta}\rfloor$\\
16:&\quad\quad\quad\textbf{else}\\
17:&\quad\quad\quad\quad$i=i+1$\\
18:&\textbf{Output:} $  \estime $ and $\mathfrak f(g)$ for all $g\in\widetilde G$.\\
\bottomrule
\end{tabular}
\bigskip

\begin{theorem}\label{thm2}
Under Assumptions \ref{as2} to \ref{as6}, $Q(g,\beta)$ given by \eqref{thresh}, and $\epsilon$ -- by \eqref{ep},
%$M, a$ as in Proposition \ref{pp1}, 
let
\begin{equation}\label{thresh1}
    Q_1(g,\beta)=Q(g,\beta)+\epsilon %e^{(\rho+a)\beta}ML\|g\|\beta+e^{a\beta}(e^{\beta}-1)MK\|g\|+4e^{\rho\beta}\sigma
\end{equation}
be the threshold in Algorithm 2. Let also $\mathfrak t_j$ and $\mathfrak f_j(g)$ be the outputs of Algorithm 2. Then $|\estime_j-t_j|\le \beta$ and

\begin{equation}\label{quest}
\begin{split}
|\mathfrak f_j(g) &-\langle h_{j}, g\rangle|  \\ &\le\epsilon+3e^{(3\rho+a)\beta}ML\|g\|\beta+e^{a\beta}(e^{3\rho\beta}-1)MK\|g\|+4e^{3\rho\beta}\sigma \\
&\quad+2e^{\rho\beta}Q_1(g,\beta)+\max\{v_{3,2}(h_j,g,\beta), v_{3,3}(h_j,g,\beta),v_{2,2}(h_j,g,\beta)\}
\end{split}
\end{equation}
where $v_{k,i}$, $k,i=2,3$, are defined by \eqref{v_2} so that 
$v_{k,i}(h_j,g,\beta)\rightarrow 0$ as $\beta\rightarrow 0$ (by Lemma \ref{lem_inpr2}).
In particular, for sufficiently small $\beta>0$, one has $|\mathfrak f_j(g)-\langle h_{j}, g\rangle| \le 13\sigma+4\epsilon$  as long as $\sigma$ and $\epsilon $ are not both $0$.
\end{theorem}

\begin{proof}

Suppose that we have detected the $(j-1)$-th burst in the time interval $[m\beta,(m+2)\beta)$ for some $m\in\N$. By Assumption \ref{as6}, the next nonzero burst $h_j$ must happen no sooner than $m\beta+D+4\beta,$ thus we just need to continue our detection from $(m+3+\lfloor \frac{D}{\beta}\rfloor)\beta.$ Now we simply denote $(m+3+\lfloor \frac{D}{\beta}\rfloor)$ by $n$ and analyze the  occurrence of a burst in the interval $[n\beta,(n+2)\beta)$. %for some $n\in\N$. 
To do that, we first evaluate the quantities $F_n$ and $\Delta_n=e^{\rho\beta}F_{n+1}-F_n$ from the measurements \eqref{meas}:

\[
\begin{array}{ll}
F_n&=\mathfrak m_{n+1}(\frac{g}{\beta})-\mathfrak m_n(\frac{T^*(\beta)g}{\beta})\\
&=\int_{\tau_n}^{(n+1)\beta}\langle T((n+1)\beta-s)f_n \phi(s-\tau_n), \frac{g}{\beta}\rangle ds+\sum\limits_{\tau_i<n\beta}\int_0^{\beta}\langle T(\beta-s)f_i \phi(n\beta+s-\tau_i), \frac{g}{\beta}\rangle ds\\

&\quad+\int_0^{\beta}\langle T(\beta-s)\eta(n\beta+s), \frac{g}{\beta}\rangle ds+\nu((n+1)\beta, \frac{g}{\beta})-\nu(n\beta, \frac{T^*(\beta)g}{\beta}),\\
\end{array}
\]
and 
\[
\begin{array}{ll}
\Delta_n&=e^{\rho\beta}F_{n+1}-F_n\\
&=e^{\rho\beta}\int_{\tau_{n+1}}^{(n+2)\beta}\langle T((n+2)\beta-s)f_{n+1}\phi(s-\tau_{n+1}), \frac{g}{\beta}\rangle ds\\

&\quad+\int_0^{\beta}\langle T(\beta-s)f_ne^{\rho\beta}\phi((n+1)\beta+s-\tau_{n}), \frac{g}{\beta}\rangle ds-\int_{\tau_n}^{(n+1)\beta}\langle T((n+1)\beta-s)f_n\phi(s-\tau_n), \frac{g}{\beta}\rangle ds\\

&\quad+\sum\limits_{\tau_i<n\beta}\int_0^{\beta}\langle T(\beta-s)f_i(e^{\rho\beta}\phi((n+1)\beta+s-\tau_i)-\phi(n\beta+s-\tau_i)), \frac{g}{\beta}\rangle ds\\

&\quad+e^{\rho\beta}\int_0^{\beta}\langle T(\beta-s)\eta((n+1)\beta+s), \frac{g}{\beta}\rangle ds-\int_0^{\beta}\langle T(\beta-s)\eta(n\beta+s), \frac{g}{\beta}\rangle ds+\alpha_n\\
\end{array}
\]
where $\alpha_n=e^{\rho\beta}\nu((n+2)\beta, \frac{g}{\beta})-e^{\rho\beta}\nu((n+1)\beta, \frac{T^*(\beta)g}{\beta})-\nu((n+1)\beta, \frac{g}{\beta})+\nu(n\beta, \frac{T^*(\beta)g}{\beta}).$\\

From the expression above, we note that since we don't have a concrete formula for $\phi(t)$, we are unable to use the technique in Subsection \ref{model1} to cancel the effect of the bursts that occurred prior to $n\beta$. However, by Assumption \ref{as6}, the requirement that the distance $|t_{j+1}-t_j|$ between two bursts  is large enough, ensures that if no burst occurred in $[n\beta, (n+2)\beta)$ (i.e. $f_n=f_{n+1}=0$), then $|e^{\rho\beta}F_{n+1}-F_n|$ is below our chosen threshold \eqref{thresh1}. We will show that via the calculations below, where we use \eqref{Thresholderr}, \eqref{condition1} and Assumption \ref{as6}. %, and assuming that no burst occur in $[n\beta, (n+2)\beta)$, we obtain

\begin{equation}
\begin{array}{ll}\label{threshold2}
\lvert\Delta_n\rvert&=|e^{\rho\beta}F_{n+1}-F_n|\\
&=|\sum\limits_{\tau_i<n\beta}\int_0^{\beta}\langle T(\beta-s)f_i(e^{\rho\beta}\phi((n+1)\beta+s-\tau_i)-\phi(n\beta+s-\tau_i)), \frac{g}{\beta}\rangle ds\\

&\quad+\int_0^{\beta}\langle T(\beta-s)(e^{\rho\beta}\eta((n+1)\beta+s)-\eta(n\beta+s)), \frac{g}{\beta}\rangle ds+\alpha_n|\\
&\le \sum\limits_{\tau_i<n\beta}|\int_0^{\beta}\langle T(\beta-s)f_i(e^{\rho\beta}\phi((n+1)\beta+s-\tau_i)-\phi(n\beta+s-\tau_i)), \frac{g}{\beta}\rangle ds|\\

&\quad+|\int_0^{\beta}\langle T(\beta-s)(e^{\rho\beta}\eta((n+1)\beta+s)-\eta(n\beta+s)), \frac{g}{\beta}\rangle ds|+|\alpha_n|\\

&\le\sum\limits_{\tau_i<n\beta}\int_0^{\beta}|e^{\rho\beta}\phi((n+1)\beta+s-\tau_i)-\phi(n\beta+s-\tau_i)|\|f_i\|\|T^*(\beta-s)\frac{g}{\beta}\|ds\\
&\quad+e^{\rho\beta}CL\|g\|{\beta}+(e^{\rho\beta}-1)CK\|g\|+4e^{\rho\beta}\sigma\\

&\le\sum\limits_{\tau_i<n\beta}2e^{-\rho(n\beta-\tau_i)}C\|f_i\|\|g\|+e^{\rho\beta}CL\|g\|{\beta}+(e^{\rho\beta}-1)CK\|g\|+4e^{\rho\beta}\sigma\\

&=\sum\limits_{k=1}^{j-1}2e^{-\rho(n\beta-t_k)}C\|h_k\|\|g\|+e^{\rho\beta}CL\|g\|{\beta}+(e^{\rho\beta}-1)CK\|g\|+4e^{\rho\beta}\sigma\\

&\le \sum\limits_{k=1}^\infty 2e^{-k\rho D}CHR+e^{\rho\beta}CL\|g\|{\beta}+
(e^{\rho\beta}-1)CK\|g\|+4e^{\rho\beta}\sigma\\

&\le \frac{2}{e^{\rho D}-1}CHR+e^{\rho\beta}CL\|g\|{\beta}+(e^{\rho\beta}-1)CK\|g\|+4e^{\rho\beta}\sigma\\
&\le \epsilon+e^{\rho\beta}CL\|g\|{\beta}+(e^{\rho\beta}-1)CK\|g\|+4e^{\rho\beta}\sigma=Q_1(g,\beta).
\end{array}
\end{equation}
 Recall that in the above calculation $C=Me^{a\beta}$, 
 $\epsilon = \frac{2}{e^{\rho D}-1}CHR$ as defined by \eqref{ep}, $H$ is the upper bound constant in Assumption \ref{as3}, $L, K$ are the Lipschitz constant and the background source upper bound, respectively, in Assumption \ref{as4}, and $R = \sup_{g\in\widetilde G} \|g\|$ as in Assumption \ref{as2}.

\begin{remark}
In this case, the time difference $D$ between every pair of adjacent non-zero bursts will influence the error estimate. When $\epsilon<\sigma$, 
%then it implies that 
the past bursts only have a very weak impact on the subsequent bursts and their influence together is even smaller than the noise level $\sigma.$ 
\end{remark}

We now assume that the $j$-th burst occurred in the interval $[(n+1)\beta, (n+2)\beta)$. Similarly to our discussion in Subsection \ref{model1}, we consider two cases:  1) the burst is detected; and 2) the burst is not detected. As before, for Case 1, the burst is detected if and only if $\lvert \Delta_n\rvert>Q_1(g,\beta)$ or $\lvert \Delta_{n+1}\rvert>Q_1(g,\beta)$. We need to establish \eqref{quest} for each of the cases (assuming $\mathfrak f_j(g)=0$ in Case 2).

%For Case 2, when the burst $h_j$ is not detected, we will show that $\langle h_j,g\rangle$ is small, and we will estimate $\langle h_j,g\rangle$. \\

\medskip\noindent 
{\it Case 1.} The $j$-th burst  is detected in $[(n+1)\beta, (n+2)\beta)$.
\medskip

Assume that $\lvert \Delta_{n}\rvert>Q_1(g,\beta)$. Then $\tau_{n+1}=t_j$, $f_{n+1}=h_j$, $f_{n-1}=f_{n}=f_{n+2}=0$ and Algorithm 2 returns $\estime_j=(n+1)\beta$ and  $\mathfrak f_j(g)=e^{3\rho\beta}F_{n+2}-F_{n-1}$. 
We get 

\[
\begin{array}{ll}
&\quad e^{3\rho\beta}F_{n+2}-F_{n-1}\\

&=\int_0^{\beta}\langle T(\beta-s)h_j e^{3\rho\beta}\phi((n+2)\beta+s-t_{j}), \frac{g}{\beta}\rangle ds\\

&\quad+\sum\limits_{\tau_i<(n-1)\beta}\int_0^{\beta}\langle T(\beta-s)f_i(e^{3\rho\beta}\phi((n+2)\beta+s-\tau_i)-\phi((n-1)\beta+s-\tau_i)), \frac{g}{\beta}\rangle ds\\

&\quad+e^{3\rho\beta}\int_0^{\beta}\langle T(\beta-s)\eta((n+2)\beta+s), \frac{g}{\beta}\rangle ds-\int_0^{\beta}\langle T(\beta-s)\eta((n-1)\beta+s), \frac{g}{\beta}\rangle ds+\alpha_{n-1}^{'}\\
\end{array}
\]
where $\alpha_{n-1}^{'}=e^{3\rho\beta}\nu((n+3)\beta, \frac{g}{\beta})-e^{3\rho\beta}\nu((n+2)\beta, \frac{T^*(\beta)g}{\beta})-\nu(n\beta, \frac{g}{\beta})+\nu((n-1)\beta, \frac{T^*(\beta)g}{\beta}).$

Computing the error gives
\begin{equation}\label{error22}
\begin{array}{ll}
&\quad|\mathfrak f_j(g)-\langle h_j, g\rangle|\\
&=|e^{3\rho\beta}F_{n+2}-F_{n-1}-\langle h_{j}, g\rangle|\\
&\le|\int_0^{\beta}\langle T(\beta-s)h_je^{3\rho\beta}\phi((n+2)\beta+s-t_j), \frac{g}{\beta}\rangle ds-\langle h_j, g\rangle|\\

&\quad+|\sum\limits_{\tau_i<(n-1)\beta}\int_0^{\beta}\langle T(\beta-s)f_i(e^{3\rho\beta}\phi((n+2)\beta+s-\tau_i)-\phi((n-1)\beta+s-\tau_i)), \frac{g}{\beta}\rangle ds|\\

&\quad+|\int_0^{\beta}\langle T(\beta-s)e^{3\rho\beta}\eta((n+2)\beta+s), \frac{g}{\beta}\rangle ds-\int_0^{\beta}\langle T(\beta-s)\eta((n-1)\beta+s), \frac{g}{\beta}\rangle ds|+|\alpha_{n-1}^{'}|\\

&\le v_{3,2}(h_j, g, \beta)+\epsilon+3e^{3\rho\beta}CL\|g\|{\beta}+(e^{3\rho\beta}-1)CK\|g\|+4e^{3\rho\beta}\sigma,
\end{array}
\end{equation}
where $v_{3,2}$ is given by \eqref{v_2} and we estimated the last two terms of the first inequality  similarly to \eqref {threshold2}. 

%\begin{remark}
Now assume that $\lvert \Delta_{n}\rvert\le Q_1(g,\beta)$ and $\lvert \Delta_{n+1}\rvert>Q_1(g,\beta)$. Then Algorithm 2 returns $\estime_j=(n+2)\beta$ and $\mathfrak f_j(g)=e^{3\rho\beta}F_{n+3}-F_{n}$. We then have 
 \begin{equation*}
\begin{array}{ll}
&\quad|\mathfrak f_j(g)-\langle h_j, g\rangle|\\
&\le|\int_0^{\beta}\langle T(\beta-s)h_je^{3\rho\beta}\phi((n+3)\beta+s-t_j), \frac{g}{\beta}\rangle ds-\langle h_j, g\rangle|\\
&\quad+\epsilon+3e^{3\rho\beta}CL\|g\|\beta+(e^{3\rho\beta}-1)CK\|g\|+4e^{3\rho\beta}\sigma,\\
\end{array}
\end{equation*}
 where $|\int_0^{\beta}\langle T(\beta-s)h_je^{3\rho\beta}\phi((n+3)\beta+s-t_j), \frac{g}{\beta}\rangle ds-\langle h_j, g\rangle|=v_{3,3}(h_j,g,\beta)$ is given by \eqref{v_2}.
%\end{remark}

\medskip\noindent 
{\it Case 2.} The $j$-th burst  is in $[(n+1)\beta, (n+2)\beta)$, but $\langle h_j, g\rangle$ it is not detected. \\
\medskip

If the $j$-th burst occurred in $[(n+1)\beta, (n+2)\beta)$ but was not detected by Algorithm 2, we use the fact that $e^{2\rho\beta}F_{n+2}-F_n$ is small to show that $\langle h_j, g\rangle\approx 0$. 

\[
\begin{array}{ll}
&\quad e^{2\rho\beta}F_{n+2}-F_{n}\\

&=\int_0^{\beta}\langle T(\beta-s)h_j e^{2\rho\beta}\phi((n+2)\beta+s-t_{j}), \frac{g}{\beta}\rangle ds\\

&\quad+\sum\limits_{\tau_i<n\beta}\int_0^{\beta}\langle T(\beta-s)f_i(e^{2\rho\beta}\phi((n+2)\beta+s-\tau_i)-\phi(n\beta+s-\tau_i)), \frac{g}{\beta}\rangle ds\\

&\quad+e^{2\rho\beta}\int_0^{\beta}\langle T(\beta-s)\eta((n+2)\beta+s), \frac{g}{\beta}\rangle ds-\int_0^{\beta}\langle T(\beta-s)\eta(n\beta+s), \frac{g}{\beta}\rangle ds+\alpha_n^{''}\\
\end{array}
\]
where $\alpha_n^{''}=e^{2\rho\beta}\nu((n+3)\beta, \frac{g}{\beta})-e^{2\rho\beta}\nu((n+2)\beta, \frac{T^*(\beta)g}{\beta})-\nu((n+1)\beta, \frac{g}{\beta})+\nu(n\beta, \frac{T^*(\beta)g}{\beta}).$
Now we estimate $|\langle h_j, g\rangle|$:
\begin{equation}\label{error21}
    \begin{array}{ll}
    &\quad|\langle h_j, g\rangle|\\
    
    &\le|-\int_0^{\beta}\langle T(\beta-s)h_je^{2\rho\beta}\phi((n+2)\beta+s-t_{j}), \frac{g}{\beta}\rangle ds|\\
    &\quad+|\int_0^{\beta}\langle T(\beta-s)h_je^{2\rho\beta}\phi((n+2)\beta+s-t_{j}), \frac{g}{\beta}\rangle ds-\langle h_j, g\rangle|\\
    &\le|-\int_0^{\beta}\langle T(\beta-s)h_je^{2\rho\beta}\phi((n+2)\beta+s-t_{j}), \frac{g}{\beta}\rangle ds+(e^{2\rho\beta}F_{n+2}-F_n)+(e^{2\rho\beta}F_{n+2}-F_n)|\\
    &\quad+v_{2,2}(h_j,g,\beta)\\

    &\le|-\int_0^{\beta}\langle T(\beta-s)h_je^{2\rho\beta}\phi((n+2)\beta+s-t_{j}), \frac{g}{\beta}\rangle ds+(e^{2\rho\beta}F_{n+2}-F_n)|\\
    &\quad+|e^{2\rho\beta}F_{n+2}-F_n|+v_{2,2}(h_j,g,\beta)\\

    &\le \epsilon+2e^{2\rho\beta}CL\|g\|{\beta}+(e^{2\rho\beta}-1)CK\|g\|+4e^{2\rho\beta}\sigma+2e^{\rho\beta}Q_1(g,\beta)+v_{2,2}(h_j,g,\beta).\\
    
    \end{array}
\end{equation}
where we have used $|e^{\rho\beta}F_{n+2}-F_{n+1}|\le Q_1(g,\beta)$, $|e^{\rho\beta}F_{n+1}-F_{n}|\le Q_1(g,\beta)$, and  estimated the term
$|-\int_0^{\beta}\langle T(\beta-s)h_je^{2\rho\beta}\phi((n+2)\beta+s-t_{j}), \frac{g}{\beta}\rangle ds+(e^{2\rho\beta}F_{n+2}-F_n)|$ as in  \eqref {threshold2}.

The theorem is proved.
\end{proof}

\section{Simulation}\label{sec4}

In order to evaluate the performance of our algorithms, we apply them to the following specific IVP:
\begin{equation*}
\begin{cases}
\dot{u}(t)=u(t)+\sum\limits_{i=1}h_i e^{-\rho(t-t_i)}\chi_{[t_, \infty)}(t)+\eta\\
u(0)=0\\
\end{cases}
\end{equation*}
with $h_1=3\sin(x)$, $h_2=2.5\cos(x)$, $h_3=x+2$, $x\in[0,1]$ and $t_1=0.25$, $t_2=0.54$, $t_3=0.78$, $t\in[0,1]$, and one of the two different types of background source:
$\eta=xe^{-Lt}$ or $\eta=x\sin(Lt)$.

Let $g_1=1$, $g_2=x$, $g_3=x^2$ be the sensor functions and compute the ground truth $\langle h_i, g_j\rangle$ for $i,j=1,2,3.$ In the simulation, we let $\rho=1$, $L=10^{-2}$ and the noise level $\sigma=10^{-3}$. The goal is to find the burst times $\{0.25, 0.76, 1.1\}$ and compare the output $\mathfrak f_i(g_j)$ with the ground truth $\langle h_i, g_j\rangle$ ($i,j=1,2,3$) for different time steps $\beta=0.015$ and $\beta=0.01$, respectively. We acquire the measurements \eqref{meas} and use the algorithm in Subsection \ref{model1}. The results are shown in Figures $\ref{fig1}$ and $\ref{fig2}$.

\begin{figure}[b]
    \centering
    \includegraphics[width=12cm]{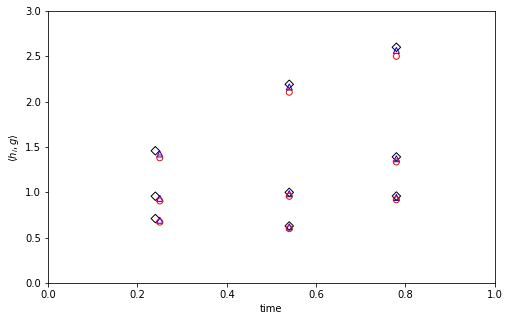}
    \caption{$\eta=xe^{-Lt}$}
\label{fig1}
\end{figure}
    
\begin{figure}[b]
    \centering    
    \includegraphics[width=12cm]{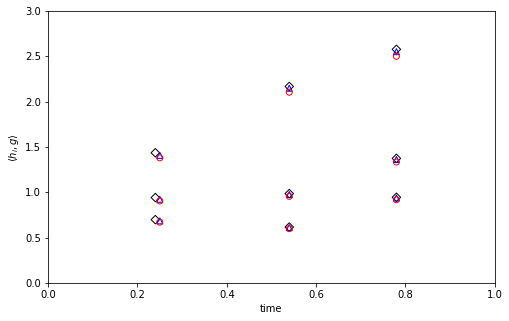}
    \caption{$\eta=xsin(Lt)$}
\label{fig2}
\end{figure}

To test the algorithms for the model in Section \ref{model2}, we use the same burst and sensor functions. We also test on the same background sources and let $L=10^{-2}.$ But here we let $\phi(t)=\frac{1}{2}(e^{-2t}+e^{-t})$, thus $0<\phi(t)\le e^{-t}.$ For other parameters, we let $t_1=1.1$, $t_2=9.8$, $t_3=19$, $D=8.6$ and $\sigma=10^{-3}$ ($\epsilon<\sigma$). The
goal is still to find out the bursts and compare the output with the ground truth for $\beta=0.015$ and $\beta=0.01$, respectively. We utilize the algorithm in Subsection \ref{model2} and the results are shown in Figures $\ref{fig3}$ and $\ref{fig4}$.

\begin{figure}[h]
    \centering
    \includegraphics[width=12cm]{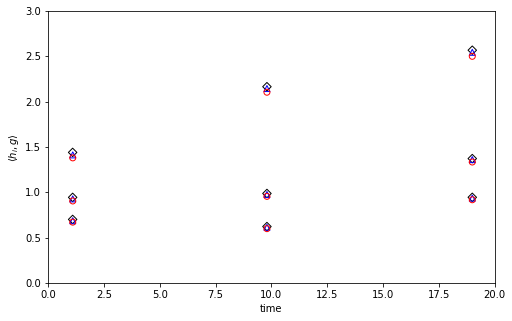}
    \caption{$\eta=xe^{-Lt}$}
\label{fig3}
\end{figure}
    
\begin{figure}[h]
    \centering    
    \includegraphics[width=12cm]{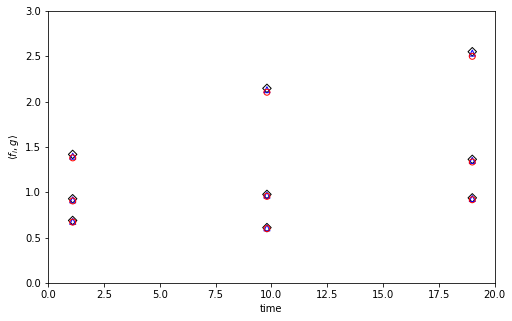}
    \caption{$\eta=xsin(Lt)$}
\label{fig4}
\end{figure}

\newpage
In Figures $\ref{fig1}$ through  $\ref{fig4}$, the results for $h_i$ lie in the $i$-th column. Red circles stand for the ground truth $\langle h_i, g_j\rangle$, black squares stand for the output $\mathfrak f_i(g_j)$ when $\beta=0.015$ and blue triangles stand for the output $\mathfrak f_i(g_j)$ when $\beta=0.01.$ As expected, the test shows that our algorithms can find out all bursts and the error gets smaller when the time step $\beta$ gets shorter. 

\newpage
\medskip 
\noindent {\bf{Acknowledgement}.} The authors of the paper were supported in part by the collaborative NSF grant DMS-2208030 and DMS-2208031.

\bibliographystyle{siam}
\bibliography{Akram_refs,refs}

\end{document}